\documentclass[12pt]{article}
\usepackage{graphicx}
\usepackage{amssymb, amsthm, amsmath, amsfonts, amsrefs}
\usepackage{mathrsfs,microtype}
\usepackage{epstopdf, color}
\usepackage{tikz}
\usepackage{hyperref}
\usetikzlibrary{arrows,decorations.pathmorphing,decorations.footprints,fadings,calc,trees,mindmap,shadows,decorations.text,patterns,positioning,shapes,matrix,fit,through}
\usepackage[margin=1.3in]{geometry}

\theoremstyle{plain}
\newtheorem{theorem}{Theorem}[section]
\newtheorem{lemma}[theorem]{Lemma}
\newtheorem{corollary}[theorem]{Corollary}

\newtheorem{question}[theorem]{Question}

\theoremstyle{definition}
\newtheorem{definition}[theorem]{Definition}

\theoremstyle{remark}
\newtheorem*{remark}{Remark}
\theoremstyle{theorem}
\newtheorem{conjecture}[theorem]{Conjecture}

\newcommand{\calG}{\mathcal{G}}

\title{Counting independent sets of a fixed size in graphs with a given minimum degree}
\author{John Engbers \and David Galvin}
\date{\today\thanks{$\{$jengbers, dgalvin1$\}$@nd.edu; Department of Mathematics,
University of Notre Dame, Notre Dame IN 46556. Galvin in part supported by National Security Agency grant H98230-10-1-0364.}}

\begin{document}

\maketitle

\begin{abstract}
Galvin showed that for all fixed $\delta$ and sufficiently large $n$, the $n$-vertex graph with minimum degree $\delta$ that admits the most independent sets is the complete bipartite graph $K_{\delta,n-\delta}$. He conjectured that except perhaps for some small values of $t$, the same graph yields the maximum count of independent sets of size $t$ for each possible $t$. Evidence for this conjecture was recently provided by Alexander, Cutler, and Mink, who showed that for all triples $(n,\delta, t)$ with $t\geq 3$, no $n$-vertex {\em bipartite} graph with minimum degree $\delta$ admits more independent sets of size $t$ than $K_{\delta,n-\delta}$.

Here we make further progress. We show that for all triples $(n,\delta,t)$ with $\delta \leq 3$ and $t\geq 3$, no $n$-vertex graph with minimum degree $\delta$ admits more independent sets of size $t$ than $K_{\delta,n-\delta}$, and we obtain the same conclusion for $\delta > 3$ and $t \geq 2\delta +1$. Our proofs lead us naturally to the study of an interesting family of critical graphs, namely those of minimum degree $\delta$ whose minimum degree drops on deletion of an edge or a vertex.
\end{abstract}

\section{Introduction and statement of results}

An independent set (a.k.a. stable set) in a graph is a set of vertices spanning no edges. For a simple, loopless finite graph $G = (V,E)$, denote by $i(G)$ the number of independent sets in $G$. In \cite{ProdingerTichy} this quantity is referred to as the {\em Fibonacci number} of $G$, motivated by the fact that for the path graph $P_n$ its value is a Fibonacci number. It has also been studied in the field of molecular chemistry, where it is referred to as the {\em Merrifield-Simmons} index of $G$ \cite{MerrifieldSimmons}.

A natural extremal enumerative question is the following: as $G$ ranges over some family $\calG$, what is the maximum value attained by $i(G)$, and which graphs achieve this maximum?
This question has been addressed for numerous families. Prodinger and Tichy \cite{ProdingerTichy} considered the family of $n$-vertex trees, and showed that the maximum is uniquely attained by the star $K_{1,n-1}$. Granville, motivated by a question in combinatorial group theory, raised the question for the family of $n$-vertex, $d$-regular graphs (see \cite{Alon} for more details). An approximate answer -- $i(G) \leq 2^{n/2(1+o(1))}$ for all such $G$, where $o(1) \rightarrow 0$ as $d \rightarrow \infty$ -- was given by Alon in \cite{Alon}, and he speculated a more exact result, that the maximizing graph, at least in the case $2d|n$, is the disjoint union of $n/2d$ copies of $K_{d,d}$. This speculation was confirmed for {\em bipartite} $G$ by Kahn \cite{KahnEntropy}, and for general regular $G$ by Zhao \cite{Zhao}. The family of $n$-vertex, $m$-edge graphs was considered by Cutler and Radcliffe in \cite{CutlerRadcliffe2}, and they observed that it is a corollary of the Kruskal-Katona theorem that the lex graph $L(n,m)$ (on vertex set $\{1, \ldots, n\}$, with edges being the first $m$ pairs in lexicographic order) maximizes $i(G)$ in this class. Zykov \cite{Zykov} considered the family of graphs with a fixed number of vertices and fixed independence number, and showed that the maximum is attained by the complement of a certain Tur\'an graph. (Zykov was actually considering cliques in a graph with given clique number, but by complementation this is equivalent to considering independent sets in a graph with given independence number.) Other papers addressing questions of this kind include \cite{Hua}, \cite{LinLin}, \cite{PedersenVestergaard} and \cite{Sapozhenko}.

Having resolved the question of maximizing $i(G)$ for $G$ in a particular family, it is natural to ask which graph maximizes $i_t(G)$, the number of independent sets of size $t$ in $G$, for each possible $t$. For many families, it turns out that the graph which maximizes $i(G)$ also maximizes $i_t(G)$ for all $t$. Wingard \cite{Wingard} showed this for trees, Zykov \cite{Zykov} showed this for graphs with a given independence number (see \cite{CutlerRadcliffe} for a short proof), and Cutler and Radcliffe \cite{CutlerRadcliffe} showed this for graphs on a fixed number of edges (again, as a corollary of Kruskal-Katona). In \cite{KahnEntropy}, Kahn conjectured that for all $2d|n$ and all $t$, no $n$-vertex, $d$-regular graph admits more independent sets of size $t$ than the disjoint union of $n/2d$ copies of $K_{d,d}$; this conjecture remains open, although asymptotic evidence appears in \cite{CarrollGalvinTetali}.

\medskip

The focus of this paper is the family $\calG(n,\delta)$ of $n$-vertex graphs with minimum degree $\delta$. One might imagine that since removing edges increases the count of independent sets, the graph in $\calG(n,\delta)$ that maximizes the count of independent sets would be $\delta$-regular (or close to), but this turns out not to be the case, even for $\delta=1$. The following result is from \cite{GalvinTwoProblems}.
\begin{theorem} \label{thm-fromGalvinTwoProblems}
For $n \geq 2$ and $G \in \calG(n,1)$, we have $i(G) \leq i(K_{1,n-1})$. For $\delta \geq 2$, $n \geq 4\delta^2$ and $G \in \calG(n,\delta)$, we have $i(G) \leq i(K_{\delta,n-\delta})$.
\end{theorem}

What about maximizing $i_t(G)$ for each $t$? The family $\calG(n,\delta)$ is an example of a family for which the maximizer of the total count is {\em not} the maximizer for each individual $t$. Indeed, consider the case $t=2$. Maximizing the number of independent sets of size two is the same as minimizing the number of edges, and it is easy to see that for all fixed $\delta$ and sufficiently large $n$, there are $n$-vertex graphs with minimum degree at least $\delta$ which have fewer edges than $K_{\delta,n-\delta}$ (consider for example a $\delta$-regular graph, or one which has one vertex of degree $\delta+1$ and the rest of degree $\delta$). However, we expect that anomalies like this occur for very few values of $t$. Indeed, the following conjecture is made in \cite{GalvinTwoProblems}.
\begin{conjecture} \label{conj-levelsetsweak}
For each $\delta \geq 1$ there is a $C(\delta)$ such that for all $t \geq C(\delta)$, $n \geq 2\delta$ and $G \in \calG(n,\delta)$, we have 
$$
i_t(G) \leq i_t(K_{\delta,n-\delta}) = \binom{n-\delta}{t} + \binom{\delta}{t}.
$$
\end{conjecture}

The case $\delta=1$ of Conjecture \ref{conj-levelsetsweak} is proved in \cite{GalvinTwoProblems}, with $C(1)$ as small as it possible can be, namely $C(1)=3$. In \cite{Cutler}, Alexander, Cutler and Mink looked at the subfamily $\calG^{\rm bip}(n,\delta)$ of {\em bipartite} graphs in $\calG(n,\delta)$, and resolved the conjecture in the strongest possible way for this family.
\begin{theorem} \label{thm-ACM}
For $\delta \geq 1$, $n \geq 2\delta$, $t \geq 3$ and $G \in \calG^{\rm bip}(n,\delta)$, we have $i_t(G) \leq i_t(K_{\delta,n-\delta})$.
\end{theorem}
This provides good evidence for the truth of the strongest possible form of Conjecture \ref{conj-levelsetsweak}, namely that we may take $C(\delta)=3$.

\medskip

The purpose of this paper is to make significant progress towards this strongest possible conjecture. We completely resolve the cases $\delta =2$ and $3$, and for larger $\delta$ we deal with all but a small fraction of cases.
\begin{theorem} \label{thm-main}
\begin{enumerate}
\item \label{item-delta=2} For $\delta =2$, $t \geq 3$ and $G \in \calG(n,2)$, we have $i_t(G) \leq i_t(K_{2,n-2})$. For $n \geq 5$ and $3 \leq t \leq n-2$ we have equality iff $G=K_{2,n-2}$ or $G$ is obtained from $K_{2,n-2}$ by joining the two vertices in the partite set of size $2$.
\item \label{item-delta=3} For $\delta =3$, $t \geq 3$ and $G \in \calG(n,3)$, we have $i_t(G) \leq i_t(K_{3,n-3})$. For $n \geq 6$ and $t=3$ we have equality iff $G=K_{3,n-3}$; for $n \geq 7$ and $4 \leq t \leq n-3$ we have equality iff $G$ is obtained from $K_{3,n-3}$ by adding some edges inside the partite set of size $3$.
\item \label{item-largedeltaweak} For $\delta \geq 3$, $t \geq 2\delta + 1$ and $G \in \calG(n,\delta)$, we have $i_t(G) \leq i_t(K_{\delta,n-\delta})$. For $n \geq 3\delta+1$ and $2\delta+1 \leq t \leq n-\delta$ we have equality iff $G$ is obtained from $K_{\delta,n-\delta}$ by adding some edges inside the partite set of size $\delta$.
\end{enumerate}
\end{theorem}
We note that part \ref{item-delta=2} above provides an alternate proof of the $\delta=2$ case of the total count of independent sets, originally proved in \cite{GalvinTwoProblems}.
\begin{corollary}
For $n \geq 4$ and $G \in \calG(n,2)$, we have $i(G) \leq i(K_{2,n-2})$. For $n=4$ and $n \geq 6$ there is equality iff $G=K_{2,n-2}$.
\end{corollary}

\begin{proof}
The result is trivial for $n=4$. For $n=5$, it is easily verified by inspection, and we find that both $C_5$ and $K_{2,3}$ have the same total number of independent sets. So we may assume $n \geq 6$.

We clearly have $i(K'_{2,n-2}) < i(K_{2,n-2})$, where $K'_{2,n-2}$ is the graph obtained from $K_{2,n-2}$ by joining the two vertices in the partite set of size $2$. For all $G \in \calG(n,2)$ different from both $K_{2,n-2}$ and $K'_{2,n-2}$, Theorem \ref{thm-main} part \ref{item-delta=2} tells us that $i_t(G) \leq i_t(K_{2,n-2})-1$ for $3 \leq t \leq n-2$. For $t=0, 1, n-1$ and $n$ we have $i_t(G) = i_t(K_{2,n-2})$ (with the values being $1$, $n$, $0$ and $0$ respectively). We have $i_2(G) \leq {n \choose 2} - n$ (this is the number of non-edges in a $2$-regular graph), and so
\begin{equation} \label{eq-2count}
i_2(G) \leq i_2(K_{2,n-2}) + {n \choose 2} - n - {n-2 \choose 2} - 1 = i_2(K_{2,n-2}) + n -4.
\end{equation}
Putting all this together we get $i(G)\leq i(K_{2,n-2})$.

If $G$ is not $2$-regular then we have strict inequality in (\ref{eq-2count}) and so $i(G) < i(K_{2,n-2})$. If $G$ is $2$-regular, then (as we will show presently) we have $i_3(G) < i_3(K_{2,n-2})-1$ and so again $i(G) < i(K_{2,n-2})$. To see the inequality concerning independent sets of size $3$ note that in any $2$-regular graph the number of independent sets of size $3$ that include a fixed vertex $v$ is the number of non-edges in the graph induced by the $n-3$ vertices $V \setminus \{v,x,y\}$ (where $x$ and $y$ are the neighbors of $v$), which is at most $\binom{n-3}{2} - (n-4)$. It follows that
\[
i_{3}(G) \leq \frac{1}{3} \left( n\left( \binom{n-3}{2} - (n-4) \right)  \right) < \binom{n-2}{3} - 1.
\]
\end{proof}

\medskip

The paper is laid out as follows. In Section \ref{sec-preliminaryremarks} we make some easy preliminary observations that will be used throughout the rest of the paper, and we introduce the important ideas of {\em ordered independent sets} and {\em critical graphs} for a given minimum degree. In Section \ref{sec-delta=2} we deal with the case $\delta=2$ (part \ref{item-delta=2} of Theorem \ref{thm-main}). We begin Section \ref{sec-largedelta} with the proof of part \ref{item-largedeltaweak} of Theorem \ref{thm-main}, and then explain how the argument can be improved (within the class of critical graphs). This improvement will be an important ingredient in the $\delta=3$ case (part \ref{item-delta=3} of Theorem \ref{thm-main}) whose proof we present in section \ref{sec-delta=3}. Finally we present some concluding remarks and conjectures in Section \ref{sec-conclusion}.

\medskip

\noindent {\bf Notation}: Throughout the paper we use $N(v)$ for the set of vertices adjacent to $v$, and $d(v)$ for $|N(v)|$. We write $u \sim v$ to indicate that $u$ and $v$ are adjacent (and $u \nsim v$ to indicate that they are not). We use $G[Y]$ to denote the subgraph induced by a subset $Y$ of the vertices, and $E(Y)$ for the edge set of this subgraph. Finally, for $t \in {\mathbb N}$ we use $x^{\underline{t}}$ to indicate the falling power $x(x-1)\ldots (x-(t-1))$.

\section{Preliminary remarks} \label{sec-preliminaryremarks}

For integers $n$, $\delta$ and $t$, let $P(n,\delta,t)$ denote the statement that for every $G \in \calG(n,\delta)$, we have $i_t(G) \leq i_t(K_{\delta, n-\delta})$. An important observation is that if we prove $P(n,\delta,t)$ for some triple $(n,\delta, t)$ with $t \geq \delta +1$, we automatically have $P(n,\delta,t+1)$. The proof introduces the important idea of considering {\em ordered} independent sets, that is, independent sets in which an order is placed on the vertices.
\begin{lemma}\label{reduction to fixed size}
For $\delta \geq 2$ and $t \geq \delta + 1$, if $G \in \calG(n,\delta)$ satisfies $i_t(G) \leq i_t(K_{\delta,n-\delta})$ then $i_{t+1}(G) \leq i_{t+1}(K_{\delta,n-\delta})$. Moreover, if $t < n-\delta$ and $i_t(G) < i_t(K_{\delta,n-\delta})$ then $i_{t+1}(G) < i_{t+1}(K_{\delta,n-\delta})$.
\end{lemma}

\begin{corollary}\label{reduction to fixed size-cor}
For $\delta \geq 2$ and $t \geq \delta + 1$, $P(n,\delta,t) \Rightarrow P(n,\delta,t+1)$.
\end{corollary}

\begin{proof}
Fix $G \in \calG(n,\delta)$. By hypothesis, the number of ordered independent sets in $G$ of size $t$ is at most $(n-\delta)^{\underline{t}}$. For each ordered independent set of size $t$ in $G$ there are at most $n-(t+\delta)$ vertices that can be added to it to form an ordered independent set of size $t+1$ (no vertex of the independent set can be chosen, nor can any neighbor of any particular vertex in the independent set). This leads to a bound on the number of ordered independent sets in $G$ of size $t+1$ of $(n-\delta)^{\underline{t}}(n-(t+\delta)) = (n-\delta)^{\underline{t+1}}$. Dividing by $(t+1)!$, we find that $i_{t+1}(G) \leq {n-\delta \choose t+1} = i_{t+1}(K_{\delta,n-\delta})$.

If we have $i_t(G) < {n-\delta \choose t}$ then we have strict inequality in the count of ordered independent sets of size $t$, and so also as long as $n-(\delta + t)>0$ we have strict inequality in the count for $t+1$, and so $i_{t+1}(G) < {n-\delta \choose t+1}$.
\end{proof}

Given Corollary \ref{reduction to fixed size-cor}, in order to prove $P(n,\delta,t)$ for $n \geq n(\delta)$ and $t \geq t(\delta)$ it will be enough to prove $P(n,\delta, t(\delta))$. Many of our proofs will be by induction on $n$, and will be considerably aided by the following simple observation.
\begin{lemma} \label{lem-good-vertex-for-deletion}
Fix $t \geq 3$. Suppose we know $P(m,\delta,t)$ for all $m < n$. Let $G \in \calG(n,\delta)$ be such that there is $v \in V(G)$ with $G-v \in \calG(n-1,\delta)$ (that is, $G-v$ has minimum degree $\delta$). Then $i_t(G) \leq i_t(K_{\delta,n-\delta})$. Equality can only occur if all of 1) $i_t(G-v) = i_t(K_{\delta,n-1-\delta})$, 2) $G-v-N(v)$ is empty (has no edges), and 3) $d(v)=\delta$ hold.
\end{lemma}

\begin{proof}
Counting first the independent sets of size $t$ in $G$ that do not include $v$ and then those that do, and bounding the former by our hypothesis on $P(m,\delta,t)$ and the latter by the number of subsets of size $t-1$ in $G-v-N(v)$, we have (with $E_k$ the empty graph on $k$ vertices)
\begin{eqnarray*}
i_t(G) & = & i_t(G-v) + i_{t-1}(G-v-N(v)) \\
 & \leq & i_t(K_{\delta,n-1-\delta}) + i_{t-1}(E_{n-1-d(v)}) \\
 & = & {n-1-\delta \choose t} + \binom{\delta}{t} + {n-1-\delta \choose t-1} \\
 & = & {n-\delta \choose t} + \binom{\delta}{t} \\
 & = & i_t(K_{\delta,n-\delta}).
\end{eqnarray*}
The statement concerning equality is evident.
\end{proof}

Lemma \ref{lem-good-vertex-for-deletion} allows us to focus on graphs with the property that each vertex has a neighbor of degree $\delta$. Another simple lemma further restricts the graphs that must be considered.
\begin{lemma} \label{lem-deltacrit}
If $G'$ is obtained from $G$ by deleting edges, then for each $t$ we have $i_{t}(G) \leq i_{t}(G')$.
\end{lemma}
This leads to the following definition.
\begin{definition}
Fix $\delta \geq 1$.  A graph $G$ with minimum degree $\delta$ is \emph{edge-critical} if for any edge $e$ in $G$, the minimum degree of $G - e$ is $\delta - 1$. It is \emph{vertex-critical} if for any vertex $v$ in $G$, the minimum degree of $G - v$ is strictly smaller than $\delta$. If it is both edge- and vertex-critical, we say that $G$ is {\em critical}.
\end{definition}
Lemmas \ref{lem-good-vertex-for-deletion} and \ref{lem-deltacrit} allow us to concentrate mostly on critical graphs. In Section \ref{sec-delta=2} (specifically Lemma \ref{structural characterization}) we obtain structural information about critical graphs in the case $\delta=2$, while much of Section \ref{sec-delta=3} is concerned with the same problem for $\delta=3$.

\medskip

An easy upper bound on the number of independent sets of size $t \geq 1$ in a graph with minimum degree $\delta$ is
\begin{equation} \label{easy-ub}
i_t(G) \leq \frac{n(n-(\delta+1))(n-(\delta+2)) \cdots (n-(\delta+(t-1)))}{t!}.
\end{equation}
This bound assumes that each vertex has degree $\delta$, and moreover that all vertices share the same neighborhood. We will obtain better upper bounds by considering more carefully when these two conditions actually hold, as having many vertices which share the same neighborhood forces those vertices in the neighborhood to have large degree.
To begin this process, it will be helpful to distinguish between vertices with degree $\delta$ and those with degree larger than $\delta$.
Set
$$
V_{=\delta} = \{v \in V(G): d(v) = \delta\}
$$
and
$$
V_{>\delta} = \{v \in V(G): d(v) > \delta\}.
$$
Most of the proofs proceed by realizing that a critical graph must have at least one of a small list of different structures in it, and we exploit the presence of a structure to significantly dampen the easy upper bound.

\section{Proof of Theorem \ref{thm-main}, part \ref{item-delta=2} ($\delta=2$)} \label{sec-delta=2}

Recall that we want to show that for $\delta =2$, $t \geq 3$ and $G \in \calG(n,2)$, we have $i_t(G) \leq i_t(K_{2,n-2})$, and that for $n \geq 5$ and $3 \leq t \leq n-2$ we have equality iff $G=K_{2,n-2}$ or $K'_{2,n-2}$ (obtained from $G$ by joining the two vertices in the partite set of size $2$). We concern ourselves initially with the inequality, and discuss the cases of equality at the end. By Corollary \ref{reduction to fixed size-cor}, it is enough to consider $t=3$, and we will prove this case by induction on $n$, the base cases $n \leq 5$ being trivial. So from here on we assume that $n > 5$ and that $P(m,2,3)$ has been established for all $m < n$, and let $G \in \calG(n,2)$ be given. By Lemmas \ref{lem-good-vertex-for-deletion} and \ref{lem-deltacrit} we may assume that $G$ is critical.

We begin with two lemmas, the first of which is well-known (see e.g. \cite{HopkinsStaton}), and the second of which gives structural information about critical graphs (in the case $\delta=2$).
\begin{lemma}\label{path-and-cycle-count}
Let $k \geq 1$ and $0 \leq t \leq k+1$.  In the $k$-path $P_k$ we have
\[
i_{t}(P_k) = \binom{k+1-t}{t}.
\]
Let $k \geq 3$ and $0 \leq t \leq k-1$.  In the $k$-cycle $C_k$ we have
\[
i_{t}(C_k) = \binom{k-t}{t} + \binom{k-t-1}{t-1}.
\]
\end{lemma}

\begin{lemma}\label{structural characterization}
Fix $\delta=2$. Let $G$ be a connected $n$-vertex critical graph. Either
\begin{enumerate}
\item $G$ is a cycle or
\item $V(G)$ can be partitioned as $Y_1 \cup Y_2$ with $2 \leq |Y_1| \leq n-3$ in such a way that $Y_1$ induces a path, $Y_2$ induces a graph with minimum degree $2$, each endvertex of the path induced by $Y_1$ has exactly one edge to $Y_2$, the endpoints of these two edges in $Y_2$ are either the same or non-adjacent, and there are no other edges from $Y_1$ to $Y_2$.
\end{enumerate}
\end{lemma}

\begin{proof}
If $G$ is not a cycle, then it has some vertices of degree greater than $2$. If there is exactly one such vertex, say $v$, then by parity considerations $d(v)$ is even and at least $4$. Since all degrees are even, the edge set may be partitioned into cycles. Take any cycle through $v$ and remove $v$ from it to get a path whose vertex set can be taken to be $Y_1$.

There remains the case when $G$ has at least two vertices with degree larger than $2$. Since $G$ is edge-critical, $V_{>\delta}$ forms an independent set and so there is a path on at least $3$ vertices joining distinct vertices $v_1, v_2 \in V_{>\delta}$, all of whose internal vertices $u_1, \ldots, u_k$ have degree $2$ (the shortest path joining two vertices in $V_{>\delta}$ would work). Since $G$ is vertex-critical we must in fact have $k \geq 2$, since otherwise $u_1$ would be a vertex whose deletion leaves a graph with minimum degree $2$. We may now take $Y_1 = \{u_1, \ldots, u_k\}$. Note that the $Y_2$ endpoints ($v_1$ and $v_2$) of the edges from $u_1$ and $u_k$ to $Y_2$ are both in $V_{>\delta}$ and so are non-adjacent.
\end{proof}

Armed with Lemmas \ref{path-and-cycle-count} and \ref{structural characterization} we now show that for critical $G$ we have
$$
i_3(G) < i_3(K_{2,n-2}) = \binom{n-2}{3}.
$$

If $G$ is the $n$-cycle, then we are done by Lemma \ref{path-and-cycle-count}. If $G$ is a disjoint union of cycles, then choose one such, of length $k$, and call its vertex set $Y_1$, and set $Y_2=V\setminus Y_1$. We will count the number of independent sets of size $3$ in $G$ by considering how the independent set splits across $Y_1$ and $Y_2$.

By Lemma \ref{path-and-cycle-count}, there are  $\binom{k-3}{3} + \binom{k-4}{2}$ independent sets of size 3 in $Y_1$ (note that this is still a valid upper bound when $k = 3$), and by induction there are at most
$\binom{n-k-2}{3}$ independent sets of size 3 in $Y_2$. There are $\left(\binom{k-1}{2} - 1\right)\left(n-k\right)$
independent sets with two vertices in $Y_1$ and one in $Y_2$ (the first factor here simply counting the number of non-edges in a $k$-cycle). Finally, there are
$k\left(\binom{n-k-1}{2} - 1\right)$
independent sets with one vertex in $Y_1$ and two in $Y_2$ (the second factor counting the number of non-edges in a $2$-regular graph on $n-k$ vertices). The sum of these bounds is easily seen to be $\binom{n-2}{3}-k$, so strictly smaller than $\binom{n-2}{3}$.

We may now assume that $G$ has a component that is not $2$-regular. Choose one such component. Let $Y_1$ be as constructed in Lemma \ref{structural characterization} and let $Y_2$ be augmented by including the vertex sets of all other components. Denote by $v_1$, $v_2$ the neighbors in $Y_2$ of the endpoints of the path. Note that it is possible that $v_1 = v_2$, but if not then by Lemma \ref{structural characterization} we have $v_1 \nsim v_2$.  We will again upper bound $i_3(G)$ by considering the possible splitting of independent sets across $Y_1$ and $Y_2$.

By Lemma \ref{path-and-cycle-count}, there are $\binom{k-2}{3}$
independent sets of size 3 in $Y_1$, and by induction there are at most
$\binom{n-k-2}{3}$ independent sets of size 3 in $Y_2$.

The number of independent sets of size $3$ in $G$ that have two vertices in $Y_1$ and one in $Y_2$ is at most
\begin{equation*}
\binom{k-3}{2}\left( n-k \right) + \left(\binom{k-1}{2} - \binom{k-3}{2} \right) (n-k-1).
\end{equation*}
The first term above counts those independent sets in which neither endpoint of the $k$-path is among the two vertices from $Y_1$, and uses Lemma \ref{path-and-cycle-count}. The second term upper bounds the number of independent sets in which at least one endpoint of the $k$-path is among the two vertices from $Y_1$, and again uses Lemma \ref{path-and-cycle-count}. (Note that when $k=2$ the application of Lemma \ref{path-and-cycle-count} is not valid, since when we remove the endvertices we are dealing with a path of length $0$, outside the range of validity of the lemma; however, the above bound is valid for $k=2$, since it equals $1$ in this case.)
Finally, the number of independent sets of size $3$ in $G$ that have one vertex in $Y_1$ and two in $Y_2$ is at most
$$
\left((k-2)\left(\binom{n-k}{2} - |E(Y_2)|\right)\right) + \sum_{i=1}^2 \left(\binom{n-k-1}{2} - |E(Y_2)| + d_{Y_2}(v_i)\right).
$$
The first term here counts the number of independent sets in which the one vertex from $Y_1$ is not an endvertex, the second factor being simply the number of non-edges in $G[Y_2]$. The second term counts those with the vertex from $Y_1$ being the neighbor of $v_i$, the second factor being the number of non-edges in $G[Y_2]-v_i$.

The sum of all of these bounds, when subtracted from $\binom{n-2}{3}$, simplifies to
\begin{equation}\label{t=3pathcount}
-(k-1)n + k^2 + k - 3 + k|E(Y_2)| - d_{Y_2}(v_1) - d_{Y_2}(v_2),
\end{equation}
a quantity which we wish to show is strictly positive.

Suppose first that $Y_1$ can be chosen so that $v_1 \neq v_2$. Recall that in this case $v_1 \nsim v_2$, so $d_{Y_2}(v_1) + d_{Y_2}(v_2) \leq |E(Y_2)|$.  Combining this with $|E(Y_2)| \geq n-k$ we get that
(\ref{t=3pathcount}) is at most $2k-3$, which is indeed strictly positive for $k \geq 2$.

If $v_1=v_2=v$, then we first note that
$$
|E(Y_2)| = \frac{1}{2}\sum_{w \in Y_2} d_{Y_2}(w) \geq \frac{d_{Y_2}(v)}{2} +(n-k-1)
$$
(since $G[Y_2]$ has minimum degree $2$). Inserting into (\ref{t=3pathcount}) we find that (\ref{t=3pathcount}) is at most
\begin{equation} \label{v1=v2}
n-3 + \left(\frac{k}{2} - 2\right)d_{Y_2}(v).
\end{equation}
This is clearly strictly positive for $k\geq 4$, and for $k=3$ strict positivity follows from $d_{Y_2}(v) < 2(n-3)$, which is true since in fact $d_{Y_2}(v) < n-3$ in this case.

If $k=2$, then (\ref{v1=v2}) is strictly positive unless $d_{Y_2}=n-3$ (the largest possible value it can take in this case). There is just one critical graph $G$ with the property that for all possible choices of $Y_1$ satisfying the conclusions of Lemma \ref{structural characterization} we have $|Y_1|=2$, $v_1=v_2=v$ and $d_{Y_2}(v)=n-3$; this is the windmill graph (see Figure \ref{fig:windmill}) consisting of $(n-1)/2$ triangles with a single vertex in common to all the triangles, and otherwise no overlap between the vertex sets (note that the degree condition on $v$ forces $G$ to be connected). A direct count gives $(n-1)(n-3)(n-5)/6 < \binom{n-2}{3}$ independent sets of size 3 in this particular graph.
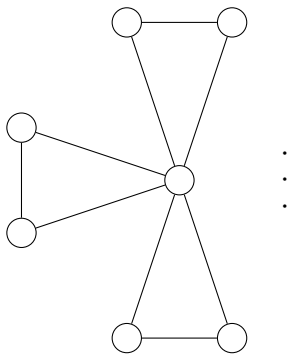
\begin{figure}[ht!]
\begin{center}
\begin{tikzpicture}[scale=.7]
	\node (v1) at (3,3) [circle,draw] {};
	\node (v2) at (2,0) [circle,draw] {};
	\node (v3) at (4,0) [circle,draw] {};
	\node (v4) at (0,2) [circle,draw] {};
	\node (v5) at (0,4) [circle,draw] {};
	\node (v6) at (2,6) [circle,draw] {};
	\node (v7) at (4,6) [circle,draw] {};
	\node at (5,2.5) {$\cdot$};
	\node at (5,3) {$\cdot$};
	\node at (5,3.5) {$\cdot$};

	\foreach \from/\to in {v1/v2,v1/v3,v1/v4,v1/v5,v1/v6,v1/v7,v2/v3,v4/v5,v6/v7}
	\draw (\from) -- (\to);
\end{tikzpicture}
\caption{The windmill graph.\label{fig:windmill}}
\end{center}
\end{figure}

This completes the proof that $i_t(G) \leq i_t(K_{2,n-2})$ for all $t \geq 3$ and $G \in \calG(n,2)$. We now turn to considering the cases where equality holds in the range $n \geq 5$ and $3 \leq t \leq n-2$. For $t=3$ and $n=5$, by inspection we see that we have equality iff $G=K_{2,3}$ or $K'_{2,3}$ (obtained from $K_{2,3}$ by adding an edge inside the partite set of size $2$). For larger $n$, we prove by induction that equality can be achieved only for these two graphs. If a graph $G$ is not edge-critical, we delete edges until we obtain a graph $G'$ which is edge-critical, using Lemma \ref{lem-deltacrit} to get $i_t(G) \leq i_t(G')$.  If $G'$ is critical, then the discussion in this section shows that we cannot achieve equality.  If $G'$ is not vertex-critical, Lemma \ref{lem-good-vertex-for-deletion} and our induction hypothesis shows that we only achieve equality for $G'$ if there is $v \in V(G')$ with $G'-v = K_{2,n-3}$ or $K'_{2,n-3}$, $G'-v-N(v)$ empty, and $d(v)=2$.  First, notice that $G'-v = K'_{2,n-3}$ implies that $G'$ is not edge-critical, so equality can only occur when $G'-v = K_{2,n-3}$.  If $G'-v=K_{2,n-3}$, the second and third conditions tell us that $N(v)$ is exactly the partite set of size $2$ in $K_{2,n-3}$, that is, that $G'=K_{2,n-2}$.  From here it is evident that equality can only occur for $G = K_{2,n-2}$ or $K'_{2,n-2}$.

Now for each fixed $n\geq 5$, we conclude from Lemma \ref{reduction to fixed size} that for $3 \leq t \leq n-2$ we {\em cannot} have equality unless $G=K_{2,n-2}$ or $K'_{2,n-2}$; and since the equality is trivial for these two cases, the proof is complete.

\section{Proof of Theorem \ref{thm-main}, part \ref{item-largedeltaweak} ($\delta \geq 3$)} \label{sec-largedelta}

Throughout this section we set $h = |V_{>\delta}|$ and $\ell = |V_{=\delta}|$; note that $h+\ell=n$. We begin this section with the proof of Theorem \ref{thm-main} part \ref{item-largedeltaweak}; we then show how the method used may be improved to obtain a stronger result within the class of critical graphs (Lemma \ref{lem-largedeltastrong} below), a result which will play a role in the proof of Theorem \ref{thm-main}, part \ref{item-delta=3} ($\delta=3$) that will be given in Section \ref{sec-delta=3}.

\medskip

Recall that we are trying to show that for $\delta \geq 3$, $t \geq 2\delta + 1$ and $G \in \calG(n,\delta)$, we have $i_t(G) \leq i_t(K_{\delta,n-\delta})$, and that for $n \geq 3\delta+1$ and $2\delta + 1 \leq t \leq n-\delta$ there is equality iff $G$ is obtained from $K_{\delta,n-\delta}$ by adding some edges inside the partite set of size $\delta$. As with Theorem \ref{thm-main} part \ref{item-delta=2} we begin with the inequality and discuss cases of equality at the end.

By Corollary \ref{reduction to fixed size-cor} it is enough to consider $t=2\delta+1$. We prove $P(n,\delta, 2\delta+1)$ by induction on $n$. For $n < 3\delta +1$ the result is trivial, since in this range all $G \in \calG(n,\delta)$ have $i_t(G)=0$. It is also trivial for $n=3\delta+1$, since the only graphs $G$ in $\calG(n,\delta)$ with $i_t(G)>0$ in this case are those that are obtained from $K_{\delta,n-\delta}$ by the addition of some edges inside the partite set of size $\delta$, and all such $G$ have $i_t(G)=1=i_t(K_{\delta,n-\delta})$. So from now on we assume $n \geq 3\delta+2$ and that $P(m,\delta,2\delta+1)$ is true for all $m < n$, and we seek to establish $P(n,\delta,2\delta+1)$.

By Lemmas \ref{lem-good-vertex-for-deletion} and \ref{lem-deltacrit} we may restrict attention to $G$ which are critical (for minimum degree $\delta$). To allow the induction to proceed, we need to show that the number of ordered independent sets of size $2\delta + 1$ in $G$ is at most $(n-\delta)^{\underline{2\delta+1}}$.

We partition ordered independent sets according to whether the first vertex is in $V_{>\delta}$ or in $V_{=\delta}$. In the first case (first vertex in $V_{>\delta}$) there are at most
\begin{eqnarray} \label{eq-hcount}
h(n-(\delta+2))(n-(\delta+3)) \cdots (n-(3\delta+1)) & = & \frac{h}{n} \left(n (n-(\delta+2))^{\underline{2\delta}}\right) \nonumber \\
& <  & \frac{h}{n} (n-\delta)^{\underline{2\delta+1}} \label{eq-hcount}
\end{eqnarray}
ordered independent sets of size $2\delta+1$, since once the first vertex has been chosen there are at most $n-(\delta+2)$ choices for the second vertex, then at most $n-(\delta + 3)$ choices for the third, and so on.

In the second case (first vertex in $V_{=\delta}$) there are at most
$$
\ell(n-(\delta+1))(n-(\delta+2)) \cdots (n-2\delta)
$$
ways to choose the first $\delta+1$ vertices in the ordered independent set. The key observation now is that since $G$ is vertex-critical there can be at most $\delta-1$ vertices distinct from $v$ with the same neighborhood as $v$, where $v$ is the first vertex of the ordered independent set. It follows that one of choices $2$ through $\delta$ has a neighbor $w$ outside of $N(v)$. Since $w$ cannot be included in the independent set, there are at most
$$
(n-(2\delta+2))(n-(2\delta+3)) \cdots (n-(3\delta+1))
$$
choices for the final $\delta$ vertices. Combining these bounds, there are at most
$$
\frac{\ell}{n} \left(n \frac{(n-(\delta+1))^{\underline{2\delta+1}}}{n-(2\delta+1)}\right) < \frac{\ell}{n} (n-\delta)^{\underline{2\delta+1}}
$$
ordered independent sets of size $2\delta+1$ that begin with a vertex from $V_{=\delta}$. Combining with (\ref{eq-hcount}) we get $i_{2\delta_+1}(G) < (n-\delta)^{\underline{2\delta+1}}/(2\delta+1)!$, as required.

This completes the proof that $i_t(G) \leq i_t(K_{\delta,n-\delta})$ for all $t \geq 2\delta+1$ and $G \in \calG(n,\delta)$. We now turn to considering the cases where equality holds in the range $n \geq 3\delta+1$ and $2\delta+1 \leq t \leq n-\delta$. For $t=2\delta+1$ and $n=3\delta+1$, we clearly have equality iff $G$ is obtained from $K_{\delta,2\delta+1}$ by adding some edges inside the partite set of size $\delta$. For larger $n$, we prove by induction that equality can be achieved only for a graph of this form. If a graph $G$ is not edge-critical, we delete edges until we obtain a graph $G'$ which is edge-critical, using Lemma \ref{lem-deltacrit} to get $i_{t}(G) \leq i_{t}(G')$.  If $G'$ is critical, then the discussion in this section shows that we cannot achieve equality.  If $G'$ is not vertex-critical, Lemma \ref{lem-good-vertex-for-deletion} and our induction hypothesis shows that we only achieve equality for $G'$ if there is $v \in V(G')$ with $G'-v$ obtained from $K_{\delta,n-\delta-1}$ by adding some edges inside the partite set of size $\delta$, $G'-v-N(v)$ empty, and $d(v)=\delta$.  First, notice that the cases where $G'-v \neq K_{\delta,n-\delta-1}$ imply that $G'$ is not edge-critical, so in fact equality can only occur when $G'-v = K_{\delta,n-\delta-1}$.  Since $d(v)=\delta$ the neighborhood of $v$ cannot include all of the partite set of size $n-1-\delta$. If it fails to include a vertex of the partite set of size $\delta$, there must be an edge in $G-v-N(v)$; so in fact, $N(v)$ is exactly the partite set of size $\delta$ and $G' = K_{\delta,n-\delta}$.  From here it is evident that equality can only occur for $G$ obtained from $K_{\delta,n-\delta}$ by adding some edges inside the partite set of size $\delta$.

Now for each fixed $n\geq 3\delta+1$, we conclude from Lemma \ref{reduction to fixed size} that for $2\delta+1 \leq t \leq n-\delta$ we {\em cannot} have equality unless $G$ is obtained from $K_{\delta,n-\delta}$ by adding some edges inside the partite set of size $\delta$; and since the equality is trivial in  these cases, the proof is complete.

\medskip

The ideas introduced here to bound the number of ordered independent sets in a critical graph, can be modified to give a result that covers a slightly larger range of $t$, at the expense of requiring $n$ to be a little larger. Specifically we have the following.
\begin{lemma} \label{lem-largedeltastrong}
For all $\delta \geq 3$, $t \geq \delta + 1$, $n \geq 3.2\delta$ and vertex-critical $G \in \calG(n,\delta)$, we have $i_t(G) < i_t(K_{\delta,n-\delta})$. For $\delta=3$ and $t=4$ we get the same conclusion for vertex-critical $G \in \calG(n,3)$ with $n \geq 8$.
\end{lemma}

\begin{remark}
The constant 3.2 has not been optimized here, but rather chosen for convenience.
\end{remark}

\begin{proof}
By Lemma \ref{reduction to fixed size} it is enough to consider $t=\delta+1$. The argument breaks into two cases, depending on whether $G$ has at most $\delta-2$ vertices with degree larger than $m$ (a parameter to be specified later), or at least $\delta-1$. The intuition is that in the former case, after an initial vertex $v$ has been chosen for an ordered independent set, many choices for the second vertex should have at least two neighbors outside of $N(v)$, which reduces subsequent options, whereas in the latter case, an initial choice of one of the at least $\delta-1$ vertices with large degree should lead to few ordered independent sets.

First suppose that $G$ has at most $\delta-2$ vertices with degree larger than $m$. Just as in (\ref{eq-hcount}), a simple upper bound on the number of ordered independent sets of size $t$ whose first vertex is in $V_{>\delta}$ is
\begin{equation} \label{eq-hcount2}
\frac{h}{n} \left(n(n-(\delta+2))(n-(\delta+3)) \cdots (n-(2\delta+1))   \right) < \frac{h}{n} (n-\delta)^{\underline{\delta+1}}.
\end{equation}
There are $\ell$ choices for the first vertex $v$ of an ordered independent set that begins with a vertex from $V_{=\delta}$. For each such $v$, we consider the number of extensions to an ordered independent set of size $\delta+1$. This is at most
\begin{equation} \label{second-vertex}
x(n-(\delta+2))^{\underline{\delta-1}} + y(n-(\delta+3))^{\underline{\delta-1}} +  z(n-(\delta+4))^{\underline{\delta-1}}
\end{equation}
where $x$ is the number of vertices in $V(G)\setminus (\{v\}\cup N(v))$ that have no neighbors outside $N(v)$, $y$ is the number with one neighbor outside $N(v)$, and $z$ is the number with at least $2$ neighbors outside $N(v)$. Note that $x+y+z=n-\delta-1$, and that by vertex-criticality $x \leq \delta-1$.

Let $u_1$ and $u_2$ be the two lowest degree neighbors of $v$. By vertex-criticality and our assumption on the number of vertices with degree greater than $m$, the sum of the degrees of $u_1$ and $u_2$ is at most $\delta+m$. Each vertex counted by $y$ is adjacent to either $u_1$ or $u_2$, so counting edges out of $u_1$ and $u_2$ there are at most $m+\delta -2x - 2$ such vertices.

For fixed $x$ we obtain an upper bound on (\ref{second-vertex}) by taking $y$ as large as possible, so we should take $y=m+\delta -2x - 2$ and $z=n-m-2\delta+x+1$. With these choices of $y$ and $z$, a little calculus shows us that we obtain an upper bound by taking $x$ as large as possible, that is, $x=\delta-1$. This leads to an upper bound on the number of ordered independent sets of size $t$ whose first vertex is in $V_{=\delta}$ of
$$
\ell\left(
\begin{array}{l}
(\delta-1)(n-(\delta+2))^{\underline{\delta-1}} + \\ (m-\delta)(n-(\delta+3))^{\underline{\delta-1}} + \\ (n-m-\delta)(n-(\delta+4))^{\underline{\delta-1}}
\end{array}
\right).
$$
Combining with (\ref{eq-hcount2}) we see that are done (for the case $G$ has at most $\delta-2$ vertices with degree larger than $m$) as long as we can show that the expression above is strictly less than $\ell(n-\delta)^{\underline{\delta+1}}/n$, or equivalently that
\begin{equation} \label{analysis1}
n\left(
\begin{array}{l}
(\delta-1)(n-(\delta+2))(n-(\delta+3)) + \\
(m-\delta)(n-(\delta+3))(n-(2\delta+1)) + \\ (n-m-\delta)(n-(2\delta+1))(n-(2\delta+2))
\end{array}
\right)
< \begin{array}{l}
(n-\delta)^{\underline{4}}
\end{array}.
\end{equation}

We will return to this presently; but first we consider the case where $G$ has at least $\delta-1$ vertices with degree larger than $m$. An ordered independent set of size $\delta+1$ in this case either begins with one of $\delta-1$ vertices of largest degree, in which case there are strictly fewer than $(n-m-1)^{\underline{\delta}}$ extensions, or it begins with one of the remaining $n-\delta+1$ vertices. For each such vertex $v$ in this second case, the second vertex chosen is either one of the $k=k(v) \leq \delta-1$ vertices that have the same neighborhood as $v$, in which case there are at most $(n-(\delta+2))^{\underline{\delta-1}}$ extensions, or it is one of the $n-d(v)-1-k$ vertices that have a neighbor that is not a neighbor of $v$, in which case there are at most $(n-(\delta+3))^{\underline{\delta-1}}$ extensions. We get an upper bound on the total number of extensions in this second case (starting with a vertex not among the $\delta -1$ of largest degree) by taking $k$ as large as possible and $d(v)$ as small as possible; this leads to a strict upper bound on the number of ordered independent sets of size $\delta+1$ in the case $G$ has at least $\delta-1$ vertices with degree larger than $m$ of
$$
(\delta-1)(n-m-1)^{\underline{\delta}} + (n-\delta+1)\left(
\begin{array}{l}
(\delta-1)(n-(\delta+2))^{\underline{\delta-1}} + \\
(n-2\delta)(n-(\delta+3))^{\underline{\delta-1}}
\end{array}
\right).
$$
We wish to show that this is at most $(n-\delta)^{\underline{\delta+1}}$. As long as $m \geq \delta$ we have $n-m-i \leq n-\delta -i$, and so what we want is implied by
\begin{equation} \label{analysis2}
\left(
\begin{array}{l}
(\delta-1)(n-m-1)(n-m-2)+\\
(n-\delta+1)(\delta-1)(n-(\delta+2))+\\
(n-\delta+1)(n-2\delta)(n-(2\delta+1))
\end{array}
\right)
\leq (n-\delta)^{\underline{3}}.
\end{equation}

Setting $m=n/2$, we find that for $\delta \geq 3$, both (\ref{analysis1}) and (\ref{analysis2}) hold for all $n \geq 3.2\delta$. Indeed, in both cases at $n=3.2 \delta$ the right-hand side minus the left-hand side is a polynomial in $\delta$ (a quartic in the first case and a cubic in the second) that is easily seen to be positive for all $\delta \geq 3$; and in both cases we can check that for each fixed $\delta \geq 3$, when viewed as a function of $n$ the right-hand side minus the left-hand side has positive derivative for all $n \geq 3.2\delta$. This completes the proof of the first statement. It is an easy check that both (\ref{analysis1}) and (\ref{analysis2}) hold for all $n \geq 8$ in the case $\delta =3$, completing the proof of the lemma.
\end{proof}

\section{Proof of Theorem \ref{thm-main}, part \ref{item-delta=3} ($\delta=3$)} \label{sec-delta=3}

Recall that we are trying to show that for $\delta =3$, $t \geq 3$ and $G \in \calG(n,3)$, we have $i_t(G) \leq i_t(K_{3,n-3})$, and that for $n \geq 6$ and $t=3$ we have equality iff $G=K_{3,n-3}$, while for $n \geq 7$ and $4 \leq t \leq n-3$ we have equality iff $G$ is obtained from $K_{3,n-3}$ by adding some edges inside the partite set of size $3$.

For $t=4$ and $n \geq 7$ we prove the result (including the characterization of uniqueness) by induction on $n$, with the base case $n=7$ trivial. For $n \geq 8$, Lemma \ref{lem-largedeltastrong} gives strict inequality for all vertex-critical $G$, so we may assume that we are working with a $G$ which is non-vertex-critical. Lemma \ref{lem-good-vertex-for-deletion} now gives the inequality $i_4(G) \leq i_4(K_{3,n-3})$, and the characterization of cases of inequality goes through exactly as it did for Theorem \ref{thm-main} parts \ref{item-delta=2} and \ref{item-largedeltaweak}. The result for larger $t$ (including the characterization of uniqueness) now follows from Lemma \ref{reduction to fixed size}.

For $t=3$, we also argue by induction on $n$, with the base case $n=6$ trivial. For $n \geq 7$, if $G$ is not vertex-critical then the inequality $i_3(G) \leq i_3(K_{3,n-3})$ follows from Lemma \ref{lem-good-vertex-for-deletion}, and the fact that there is equality in this case only for $G=K_{3,n-3}$ follows exactly as it did in the proofs of Theorem \ref{thm-main} parts \ref{item-delta=2} and \ref{item-largedeltaweak}. So we may assume that $G$ is vertex-critical. We will also assume that $G$ is edge-critical (this assumption is justified because in what follows we will show $i_3(G)<i_3(K_{3,n-3})$, and restoring the edges removed to achieve edge-criticality maintains the strictness of the inequality). Our study of critical $3$-regular graphs will be based on a case analysis that adds ever more structure to the $G$ under consideration. A useful preliminary observation is the following.
\begin{lemma} \label{lem-no_n-3}
Fix $\delta=3$. If a critical graph $G$ has a vertex $w$ of degree $n-3$ or greater, then $i_3(G) < i_3(K_{3,n-3})$.
\end{lemma}

\begin{proof}
If $d(w)> n-3$ then there are no independent sets of size $3$ containing $w$, and by Theorem \ref{thm-main} part \ref{item-delta=2} the number of independent sets of size $3$ in $G-w$ (a graph of minimum degree $2$) is at most $\binom{n-3}{3} < i_3(K_{3,n-3})$. If $d(w)=n-3$ and the two non-neighbors of $w$ are adjacent, then we get the same bound. If they are not adjacent (so there is one independent set of size $3$ containing $w$) and $G-w$ is not extremal among minimum degree $2$ graphs for the count of independent sets of size $3$, then we also get the same bound, since now $i_3(G-w) \leq \binom{n-3}{3}-1$. If $G-w$ is extremal it is either $K_{2,n-3}$ or $K'_{2,n-3}$, and in either case $w$ must be adjacent to everything in the partite set of size $n-3$ (to ensure that $G$ has minimum degree $3$), and then, since the non-neighbors of $w$ are non-adjacent, it must be that $G=K_{3,n-3}$, a contradiction since we are assuming that $G$ is critical.
\end{proof}

\subsection{Regular $G$} \label{subsec1}

If $G$ is $3$-regular then we have $i_3(G) < \binom{n-3}{3}+1$. We see this by considering ordered independent sets of size $3$. Given an initial vertex $v$, we extend to an ordered independent set of size $3$ by adding ordered non-edges from $V \setminus (N(v) \cup \{v\})$. Since $G$ is 3-regular there are $3n$ ordered edges in total, with at most $18$ of them adjacent either to $v$ or to something in $N(v)$. This means that the number of ordered independent sets of size 3 in $G$ is at most
$$
n((n-4)(n-5) - (3n-18)) < (n-3)(n-4)(n-5)+6
$$
with the inequality valid as long as $n \geq 7$. So from here on we may assume that $G$ is not $3$-regular, or equivalently that $V_{>3} \neq \emptyset$.

\begin{remark}
The argument above generalizes to show that $\delta$-regular graphs have at most $\binom{n-\delta}{3} + \binom{\delta}{3}$ independent sets of size $3$, with equality only possible when $n=2\delta$.
\end{remark}

Let $v \in V(G)$ have a neighbor in $V_{>\delta}$. By criticality $d(v)=3$. Let $w_1$, $w_2$, and $w_3$ be the neighbors of $v$, listed in decreasing order of degree, so $d(w_1)=d$, $d(w_2)=x$ and $d(w_3)=3$ satisfy $3 \leq x \leq d \leq n-4$, the last inequality by Lemma \ref{lem-no_n-3} as well as $d > 3$ (see Figure \ref{fig:generic}).
\begin{figure}[ht!]
\begin{center}
\begin{tikzpicture}[scale=.8]
	\node (v1) at (1,5) [circle,draw,label=180:$v$] {};
	\node (v2) at (4,7) [circle,draw,label=right:$w_1$ with degree $d>3$] {};
	\node (v3) at (4,5) [circle,draw,label=right:$w_2$ with degree $3 \leq x \leq d$] {};
	\node (v4) at (4,3) [circle,draw,label=right:$w_3$ with degree $3$] {};
    \foreach \from/\to in {v1/v2,v1/v3,v1/v4}
	\draw (\from) -- (\to);
\end{tikzpicture}
\caption{The generic situation from the end of Section \ref{subsec1} on.\label{fig:generic}}
\end{center}
\end{figure}
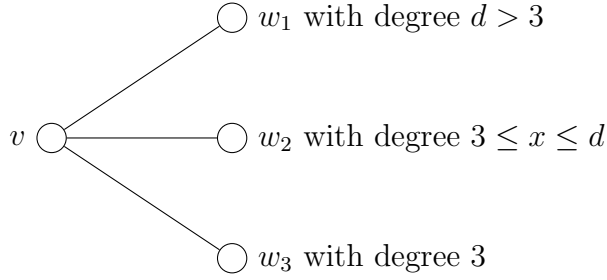

\subsection{No edge between $w_3$ and $w_2$} \label{subsec2}

We now precede by a case analysis that depends on the value of $x$ as well as on the set of edges present among the $w_i$'s. The first case we consider is $w_3 \nsim w_2$. In this case
we give upper bounds on the number of independent sets of size $3$ which contain $v$ and the number which do not.  There are $\binom{n-4}{2} - |E(Y)|$
independent sets of size 3 which include $v$, where $Y=V \setminus (N(v) \cup \{v\})$.  We lower bound $|E(Y)|$ by lower bounding the sum of the degrees in $Y$ and then subtracting off the number of edges from $Y$ to $\{v\} \cup N(v)$. This gives
\begin{equation} \label{case1a}
|E(Y)| \geq \frac{3(n-4)-2-(d-1)-(x-1)}{2} = \frac{3(n-4) - x - d}{2}.
\end{equation}
To bound the number of independent sets of size $3$ which don't include $v$, we begin by forming $G'$ from $G$ by deleting $v$ and (to restore minimum degree $3$) adding an edge between $w_3$ and $w_2$ (we will later account for independent sets that contain both $w_2$ and $w_3$). The number of independent sets of size $3$ in $G'$ is, by induction, at most $i_3(K_{3,n-4})$. But in fact, we may assume that the count is strictly smaller than this. To see this, note that if we get exactly $i_3(K_{3,n-4})$ then by induction $G' = K_{3,n-4}$. For $n=7$ this forces $G$ to have a vertex of degree $4$ and so $i_3(G)<i_3(K_{3,4})$ by Lemma \ref{lem-no_n-3}. For $n > 7$, $w_3$ must be in the partite set of size $n-4$ in $G'$ (to have degree $3$) so since $w_2 \sim w_3$ (in $G'$), $w_2$ must be in the partite set of size $3$. To avoid creating a vertex of degree $n-3$ in $G$, $w_1$ must be in the partite set of size $n-4$. But then all other vertices in the partite set of size $n-4$ only have neighbors of degree $n-4$ (in $G$), contradicting criticality.

So we may now assume that the number of independent sets of size $3$ in $G$ which do not include $v$ is at most
\begin{equation} \label{case1b}
\binom{n-4}{3} + (n-x-2),
\end{equation}
the extra $n-x-2$ being an upper bound on the number of independent sets of size $3$ that include both $w_3$ and $w_2$. Combining (\ref{case1a}) and (\ref{case1b}) we find that in this case
\begin{equation} \label{case1}
i_3(G)  \leq \binom{n-4}{2} - \frac{3(n-4) - x - d}{2} + \binom{n-4}{3} + (n-x-2).
\end{equation}
As long as $d < n+x-6$ this is strictly smaller that $i_3(K_{3,n-3})$. Since $x \geq 3$ and $d < n-3$, this completes the case $w_3 \nsim w_2$.

\subsection{Edge between $w_3$ and $w_2$, no edge between $w_3$ and $w_1$, degree of $w_2$ large} \label{subsec3}

The next case we consider is $w_3 \sim w_2$, $w_3 \nsim w_1$, and $x > 3$. In this case we can run an almost identical the argument to that of Section \ref{subsec2}, this time adding the edge from $w_1$ to $w_3$ when counting the number of independent sets of size $3$ that don't include $v$. We add $1$ to the right-hand side of (\ref{case1a}) (to account for the fact that there is now only one edge from $w_3$ to $Y$ instead of $2$, and only $x-2$ from $w_2$ to $Y$ instead of $x-1$) and replace (\ref{case1b}) with $\binom{n-4}{3} +1 + (n-d-2)$ (the $1$ since in this case we do not need strict inequality in the induction step). Upper bounding $-d$ in this latter expression by $-x$, we get the same inequality as (\ref{case1}).

\subsection{Edge between $w_3$ and $w_2$, edge between $w_3$ and $w_1$, degree of $w_2$ large} \label{subsec4}

Next we consider the case $w_3 \sim w_2$, $w_3 \sim w_1$, and $x > 3$. Here we must have $w_1 \nsim w_2$, since otherwise $G$ would not be edge-critical. The situation is illustrated in Figure \ref{fig:squarecase}.
\begin{figure}[ht!]
\begin{center}
\begin{tikzpicture}[scale=.8]
	\node (v1) at (1,1) [circle,draw,label=215:$w_3$] {};
	\node (v2) at (4,1) [circle,draw,label=-45:$w_2$] {};
	\node (v3) at (1,4) [circle,draw,label=135:$v$] {};
	\node (v4) at (4,4) [circle,draw,label=45:$w_1$] {};
	\foreach \from/\to in {v1/v2,v1/v3,v1/v4,v2/v3,v3/v4}
	\draw (\from) -- (\to);
	\draw (v2) -- (4.75,1) [dashed];
	\draw (v2) -- (4.75,1.5) [dashed];
	\draw (v4) -- (4.75,4) [dashed];
	\draw (v4) -- (4.75,3.5) [dashed];
\end{tikzpicture}
\caption{The situation in Section \ref{subsec4}.\label{fig:squarecase}}
\end{center}
\end{figure}
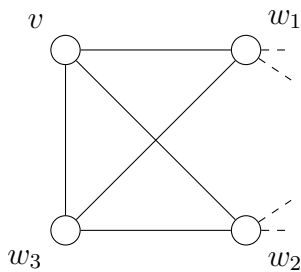
To bound $i_3(G)$, we consider $v$ and $w_3$. Arguing as in Section \ref{subsec2} (around (\ref{case1a})), the number of independent sets including one of $v$, $w_3$ is at most
\[
2\left(\binom{n-4}{2} - \frac{3(n-4) - (d-2) - (x-2)}{2}\right)
\]
To obtain an upper bound on the number of independent sets including neither $v$ nor $w_3$, we delete both vertices, add an edge from $w_1$ to $w_2$ (to restore minimum degree $3$) and use induction to get a bound of
\[
\binom{n-5}{3} + 1 + (n-d-2)
\]
(where the $n-d-2$ bounds the number of independent sets containing both $w_1$ and $w_2$). Since $x \leq n-2$ the sum of these two bound is strictly smaller than $i_3(K_{3,n-3})$.

\subsection{None of the above} \label{subsec5}

If there is no $v$ of degree $3$ that puts us into one of the previous cases, then every $v$ of degree $3$ that has a neighbor $w_1$ of degree strictly greater than $3$ may be assumed to have two others of degree $3$, $w_2$ and $w_3$ say, with $vw_2w_3$ a triangle (see Figure \ref{fig:finalcase}).
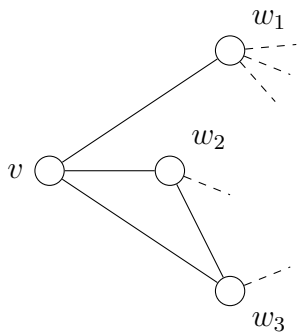
\begin{figure}[ht!]
\begin{center}
\begin{tikzpicture}[scale=.8]
	\node (v1) at (1,5) [circle,draw,label=180:$v$] {};
	\node (v2) at (4,7) [circle,draw,label=45:$w_1$] {};
	\node (v3) at (3,5) [circle,draw,label=45:$w_2$] {};
	\node (v4) at (4,3) [circle,draw,label=-45:$w_3$] {};
	\foreach \from/\to in {v1/v2,v1/v3,v1/v4,v3/v4}
	\draw (\from) -- (\to);
	\draw (v3) -- (4,4.6) [dashed];
	\draw (v4) -- (5,3.4) [dashed];
	\draw (v2) -- (5.1,7.1) [dashed];
	\draw (v2) -- (5,6.6) [dashed];
	\draw (v2) -- (4.75,6.15) [dashed];
\end{tikzpicture}
\caption{The situation in Section \ref{subsec5}.\label{fig:finalcase}}
\end{center}
\end{figure}

Since every neighbor of a vertex of degree greater than $3$ has degree exactly $3$ (by criticality) it follows that for every $w_1$ of degree greater than $3$, every neighbor of $w_1$ is a vertex of a triangle all of whose vertices have degree $3$. We claim that two of these triangles must be vertex disjoint. Indeed, if $w_1$ has two neighbors $a$ and $b$ with $a\sim b$ then the triangles associated with $a$ and $b$ must be the same, and by considering degrees we see that the triangle associated with any other neighbor of $w_1$ must be vertex disjoint from it. If $a$ and $b$ are not adjacent and their associated triangles have no vertex in common, then we are done; but if they have a vertex in common then (again by considering degrees) they must have two vertices in common, and the triangle associated with any other neighbor of $w_1$ must be vertex disjoint from both.

By suitable relabeling, we may therefore assume that $G$ has distinct vertices $w_1$ (of degree greater than $3$) and $x, y_2, y_3, v, w_2$ and $w_3$ (all of degree $3$), with $x$ and $v$ adjacent to $w_1$, and with $xy_2y_3$ and $vw_2w_3$ forming triangles (see Figure \ref{fig:finalG}). By considering degrees, we may also assume that the $w_i$'s and $y_i$'s are ordered so that $w_i \nsim y_i$ for $i=1,2$.
\begin{figure}[ht!]
\begin{center}
\begin{tikzpicture}[scale=.8]
	\node (v1) at (1,5) [circle,draw,label=180:$v$] {};
	\node (v2) at (6,8) [circle,draw,label=45:$w_1$] {};
	\node (v3) at (3,5) [circle,draw,label=45:$w_2$] {};
	\node (v4) at (4,2) [circle,draw,label=-45:$w_3$] {};
	\node (w1) at (11,5) [circle,draw,label=0:$x$] {};
	\node (w3) at (9,5) [circle,draw,label=135:$y_2$] {};
	\node (w4) at (8,2) [circle,draw,label=-135:$y_3$] {};
	\foreach \from/\to in {v1/v2,v1/v3,v1/v4,v3/v4,w1/v2,w1/w3,w1/w4,w3/w4}
	\draw (\from) -- (\to);
	\draw (v3) -- (4,4.6) [dashed];
	\draw (v4) -- (5,2.4) [dashed];
	\draw (w3) -- (8,4.6) [dashed];
	\draw (w4) -- (7,2.4) [dashed];
\end{tikzpicture}
\caption{The forced structure in Section \ref{subsec5}, before modification.\label{fig:finalG}}
\end{center}
\end{figure}
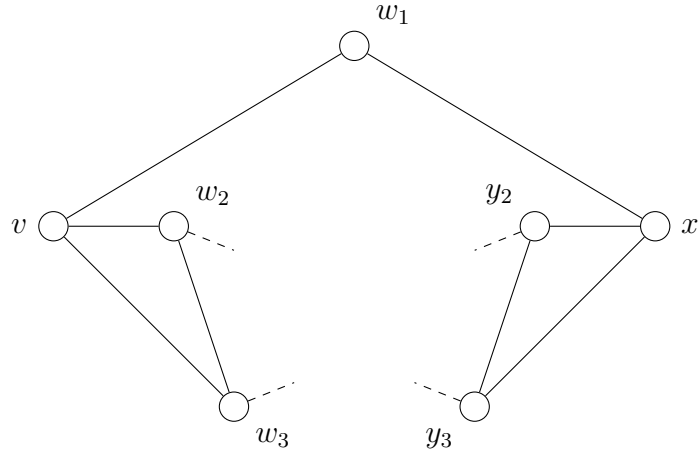

From $G$ we create $G'$ by removing the edges $w_2w_3$ and $y_2y_3$, and adding the edges $w_2y_2$ and $w_3y_3$ (see Figure \ref{fig:finalGprime}). We will argue that $i_3(G) \leq i_3(G')$; but then by the argument of Section \ref{subsec2} we have $i_3(G') < i_3(K_{3,n-3})$, and the proof will be complete.
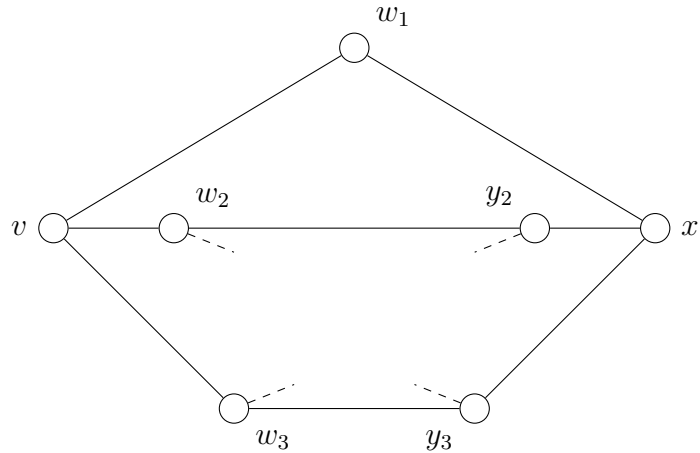
\begin{figure}[ht!]
\begin{center}
\begin{tikzpicture}[scale=.8]
	\node (v1) at (1,5) [circle,draw,label=180:$v$] {};
	\node (v2) at (6,8) [circle,draw,label=45:$w_1$] {};
	\node (v3) at (3,5) [circle,draw,label=45:$w_2$] {};
	\node (v4) at (4,2) [circle,draw,label=-45:$w_3$] {};
	\node (w1) at (11,5) [circle,draw,label=0:$x$] {};
	\node (w3) at (9,5) [circle,draw,label=135:$y_2$] {};
	\node (w4) at (8,2) [circle,draw,label=-135:$y_3$] {};
	\foreach \from/\to in {v1/v2,v1/v3,v1/v4,w1/v2,w1/w3,w1/w4,v3/w3,v4/w4}
	\draw (\from) -- (\to);
	\draw (v3) -- (4,4.6) [dashed];
	\draw (v4) -- (5,2.4) [dashed];
	\draw (w3) -- (8,4.6) [dashed];
	\draw (w4) -- (7,2.4) [dashed];
\end{tikzpicture}
\caption{The forced structure in Section \ref{subsec5}, after modification (i.e. in $G'$).\label{fig:finalGprime}}
\end{center}
\end{figure}

Independent sets of size $3$ in $G$ partition into $I_{w_2y_2}$ (those containing both $w_2$ and $y_2$, and so neither of $y_3$, $w_3$), $I_{w_3y_3}$ (containing both $w_3$ and $y_3$), and $I_{\rm rest}$, the rest. Independent sets of size $3$ in $G'$ partition into $I'_{w_2w_3}$, $I'_{y_2y_3}$, and $I'_{\rm rest}$. We have $|I_{\rm rest}|=|I'_{\rm rest}|$ (in fact $I_{\rm rest}=I'_{\rm rest}$). We will show $i_3(G) \leq i_3(G')$ by exhibiting an injection from $I_{w_2y_2}$ into $I'_{w_2w_3}$ and one from $I_{w_3y_3}$ into $I'_{y_2y_3}$.

If it happens that for every independent set $\{w_2, y_2, a\}$ in $G'$, the set $\{w_2, w_3, a\}$ is also an independent set, then we have a simple injection from $I_{w_2y_2}$ into $I'_{w_2w_3}$. There is only one way it can happen that $\{w_2, y_2, a'\}$ is an independent set but $\{w_2, w_3, a'\}$ is not; this is when $a'$ is the neighbor of $w_3$ that is not $v$ or $w_2$. If $\{w_2, y_2, a'\}$ is indeed an independent set in case, then letting $b'$ be the neighbor of $y_2$ that is not $x$ or $y_3$, we find that $\{w_2, w_3, b'\}$ is an independent set in $G'$, but $\{w_2, y_2, b'\}$ is not. So in this case we get an injection from $I_{w_2y_2}$ into $I'_{w_2w_3}$ by sending $\{w_2, y_2, a\}$ to $\{w_2, w_3, a\}$ for all $a \neq a'$, and sending $\{w_2, y_2, a'\}$ to $\{w_2, w_3, b'\}$. The injection from $I_{w_3y_3}$ into $I'_{y_2y_3}$ is almost identical and we omit the details.

\section{Concluding remarks} \label{sec-conclusion}

There now seems to be ample evidence to extend Conjecture \ref{conj-levelsetsweak} as follows.
\begin{conjecture} \label{conj-levelsetsstrong}
For each $\delta \geq 1$, $n \geq 2\delta$, $t \geq 3$ and $G \in \calG(n,\delta)$, we have $i_t(G) \leq i_t(K_{\delta,n-\delta})$.
\end{conjecture}

\medskip

Throughout we have considered $n \geq 2\delta$, that is, $\delta$ small compared to $n$. It is natural to ask what happens in the complementary range $\delta > n/2$. In the range $n \geq 2\delta$ we (conjecturally) maximize the count of independent sets by extracting as large an independent set as possible (one of size $n-\delta$). In the range $\delta > n/2$ this is still the largest independent set size, but now it is possible to have many disjoint independent sets of this size. The following conjecture seems quite reasonable.
\begin{conjecture} \label{conj-totalcount}
For $\delta \geq 1$, $n \geq \delta+1$ and $G \in \calG(n,\delta)$, we have $i(G) \leq i(K_{n-\delta,n-\delta, \ldots, n-\delta, x})$, where $0 \leq x < n-\delta$ satisfies $n \equiv x$ (mod $n-\delta$).
\end{conjecture}
\begin{question} \label{question-fixedsizecount}
For $\delta \geq 1$, $n \geq \delta+1$ and $t \geq 3$, which $G \in \calG(n,\delta)$ maximizes $i_t(G)$?
\end{question}

When $n-\delta$ divides $n$ (that is, $x=0$), both Conjecture \ref{conj-totalcount} and Question \ref{question-fixedsizecount} turn out to be easy; in this case (\ref{easy-ub}) gives that for all $1 \leq t \leq n-\delta$ and all $G \in \calG(n,\delta)$ we have $i_t(G) \leq i_t(K_{n-\delta,n-\delta, \ldots, n-\delta})$ and so also $i(G) \leq i(K_{n-\delta,n-\delta, \ldots, n-\delta})$ (the case $n=2\delta$ was observed in \cite{Cutler}). The problem seems considerably more delicate when $x \neq 0$. 

\medskip

Lemmas \ref{lem-good-vertex-for-deletion} and \ref{lem-deltacrit} allow us in the present paper to focus attention on the class of edge- and vertex-critical graphs. Lemma \ref{structural characterization} gives us a good understanding of critical graphs in the case $\delta=2$, and the bulk of Section \ref{sec-delta=3} concerns structural properties of critical graphs for $\delta = 3$. It is clear that approaching even the case $\delta =4$ by similar arguments would be considerable work. Any answer to the following question would help significantly.
\begin{question} \label{question-critical}
For $\delta \geq 4$, what can be said about the structure of edge- and vertex-critical graphs?
\end{question}

\end{document}